\documentclass{birkjour}
\usepackage{graphicx}
\usepackage{amssymb,amsmath}
\usepackage{amsthm}
\usepackage{subfigure}
\usepackage{float}
\usepackage{epsfig}
\usepackage{rotating}
\usepackage{MnSymbol}
\usepackage{qtree}
\usepackage{amsmath}
\usepackage{mathtools}
\newcommand{\U}{\mathbin{\mathcal{U}\kern-.1em}}
\renewcommand{\S}{\mathbin{\mathcal{S}\kern-.08em}}

\renewcommand{\a}{\alpha}
\renewcommand{\b}{\beta}

\renewcommand{\O}{\Omega}

\newcommand{\mf}{\mathfrak}
\newcommand{\mc}{\mathcal}

\theoremstyle{plain}% default
\newtheorem{proposition}{Proposition}[section]
\newtheorem{theorem}[proposition]{Theorem}
\newtheorem{lemma}[proposition]{Lemma}
\theoremstyle{definition}
\newtheorem{definition}[proposition]{Definition}

\newtheorem{corollary}[proposition]{Corollary}
 \theoremstyle{definition}
 
 \theoremstyle{remark}

 \numberwithin{equation}{section}

\begin{document}

%-------------------------------------------------------------------------
% editorial commands: to be inserted by the editorial office
%
%\firstpage{1} \volume{228} \Copyrightyear{2004} \DOI{003-0001}
%\seriesextraline{This is the Concrete Title of this Book\br H.E. R and S.T.C. W, Eds.}
%\firstpage{1}

\title[]
 {Curvature tensors and hyperbolic solitons on Lorentzian trans-Sasakian space form }

%----------Author 1
\author{Bidhan Mondal}

\address{%
Department of Mathematics, Jadavpur University}

\email{bidhanmondal381@gmail.com}

%\thanks{This work was completed with the support of our
%----------Author 2
\author{Nirabhra Basu}

\address{%
Department of Mathematics, Bhawanipur Education Society College}

\email{basu.nirabhra@gmail.com}
%----------Author 4
\author{Arindam Bhattacharyya}

\address{%
Department of Mathematics, Jadavpur University}

\email{arindambhat16@gmail.com}



%\date{January 1, 2024}
\dedicatory{}
%%% ----------------------------------------------------------------------

\begin{abstract}
    Lorantzian trans-Sasakian space form is a special type of space form in which the nature of even and odd dimensional space form both exist. Various curvature tensors with respect to Levi-Civita connection on the space form are derived in this paper. We have shown that if an odd-dimensional Lorentzian trans-Sasakian space form admits a hyperbolic Ricci soliton and hyperbolic conformal Ricci soliton then they will be $\eta-$Einstein. We also obtained the conditions for the solitons to be expanding, steady, or shrinking. 
Finally, an example has been constructed which justifies the results obtained.
\end{abstract}

%%% ----------------------------------------------------------------------
\maketitle
%%% ----------------------------------------------------------------------
%\tableofcontents
\section{Introduction}

In differential geometry, the Lorentzian metric is a very fascinating and important topic. In 1989, Matsumoto \cite{MK} introduced Lorentzian quasi-paracontact manifolds. After that several authors discussed different classes of Lorentzian quasi-paracontact manifolds, viz, Lorentzian Sasakian manifold, Lorentzian $\a-$ Sasakian manifold and Lorentzian $\b-$ Kemotsu manifold \cite{ITM},\cite{YAM},\cite{YAA}. In 2011, S. S. Pujar initiated the study of $\delta-$Lorentzin $\b-$Kanmotsu manifolds and $\delta-$Lorentzin $\a-$Sasakian manifolds \cite{PS}. Before these study, S. S. Pujar also introduced the notion of Lorentzian trans-Sasakian manifolds and some of its properties in \cite{PSS}. In 2013, U. C. De and K. De have found some properties of curvatures in Lorentzian trans-Sasakian manifolds \cite{UKD}.\\
If $M$ is a Lorentzian manifold then it has a Lorentzian metric $g$ which is a non-degenerate symmetric $(0,2)$ tensor field of index $1$. For index 1, the Lorentzian manifold $M$ contains not only spacelike vector fields, but also
 timelike and lightlike vector fields. A differentiable manifold $M$ admits a Lorentzian metric if and only if $M$ has a $1-$dimensional distribution. Inspired by the previous results, Bhati \cite{BSM} developed the notion of $\delta-$Lorentzian trans-Sasakian manifolds.\\
 
  In 1984, Hamilton introduced one of the most significant geometric flows, the Ricci flow \cite{CBL}. In recent years, there has been a significant increase in the study of various types of solitons as stationary solutions of various geometric flows with the goal of determining the geometrical properties of a soliton, particularly those related to its curvature and establishing obstacles for a manifold to be a soliton. The hyperbolic geometric flow is one of these geometric flows; it is a system of second-order nonlinear partial differential equations and is very close to the wave equation. It was first introduced by Kong and Liu \cite{KL} in the year 2007 to characterize the wave phenomenon of metrics and curvatures of Riemannian manifolds. The general version of hyperbolic geometric flow for the metric $g$ is given by 
\begin{equation}\label{eq001}
    \frac{\partial^2 g}{\partial t^2}+2S+\mc{F}\left(g,\frac{\partial g}{\partial t}\right)=0, \quad g(0) = g_0,
\end{equation}
where $S$ is the Ricci tensor and $\mc{F}$ is a given smooth function of the Riemannian metric
$g$ and its first order derivative with respect to $t$. Three most important special cases of this general version are \textit{standard hyperbolic geometric flow} or \textit{ hyperbolic geometric
flow}, the \textit{Einstein’s hyperbolic geometric flow} and the \textit{dissipative hyperbolic geometric flow}. In this present paper we consider the
standard hyperbolic geometric flow on a Riemann manifold associated with its Ricci curvature tensor $S$ given by: 
\begin{equation}\label{eq001}
    \frac{\partial^2 g}{\partial t^2}+2S=0\quad g(0) = g_0,\quad \frac{\partial g}{\partial t}\bigg|_{t_{0}} = k_0,
\end{equation}
where $k_0$ is a symmetric 2-tensor field on the Riemannian manifold.\\
\begin{definition}
     A Riemannian manifold $(M, g)$ is said to admit hyperbolic Ricci soliton if there exists a vector field $V$ on $M$ and real scalars $\mu$ and $\lambda$ satisfying  
     \begin{equation}
         \mc{L}_V(\mc{L}_Vg)+2\lambda\mc{L}_Vg+2S=2\mu g,
     \end{equation}
where $\mc{L}$ is the Lie derivative along the vector field $V$.
\end{definition}
It was introduced by Hamed Faraji, Shahroud Azami and Ghodratallah Fasihi-Ramandi in 2023 \cite{FHS}. The solitons are expanding, steady or shrinking according as $\lambda>0, \lambda=0$ or $\lambda <0$ respectively. When $V = 0$ or $V$ is a Killing vector field, then the manifold reduces to Einstein manifold.\\
In 2004, A. E. Fischer \cite{FA} introduced a new type of geometry flow, which preserves the scalar curvature of the evolving metrics. Because the role of conformal geometry in maintaining scalar curvature constraint, it is named conformal Ricci flow.\\
In an $n-$dimensional manifold $M$ $(n\geq 3)$, the conformal Ricci flow equation on a smooth closed connected manifold is given by
\begin{equation}\nonumber
    \frac{\partial g}{\partial t}+2(S+\frac{g}{n})=-pg\quad and \quad r(g)=-1,
\end{equation}
where $p$ is a non dynamical (time dependent) scalar field and $r(g)$ is the scalar curvature of the manifold. The term $-pg$ acts as the constraint force to maintain the scalar curvature constraint. These evolution equations are analogous to famous Navier-Stokes
equations in Fluid dynamics, where the constraint is divergence free. For this reason $p$ is also called conformal pressure.\\
In 2015, N. Basu and A. Bhattacharyya  introduced \cite{ABNBl} conformal Ricci soliton and have \cite{NAB} its existence recently. The conformal Ricci soliton equation was given by \\
$$\mc{L}_Vg+2S=2[\lambda -\frac{1}{2}(p+\frac{2}{2n+1})] g,$$
where $\lambda$ is constant and $p$ is the conformal pressure.\\

These studies make us interested to study Lorentzian trans-Sasakian manifolds and we have found some properties of curvature tensors on it. In section 3, these properties of curvature tensors have been discussed. These properties and related results help us to construct the expression of Riemannian and Ricci curvatures of Lorentzian trans-Sasakian space form. In section 4 contains the study of hyperbolic Ricci soliton and its geometry. We have included an example which satisfies the related results which have been discuss in section 5. In section 6 contains the study of  hyperbolic conformal Ricci flow, soliton, and its existence.

\section{Preliminaries}
An $(2n+1)-$dimensional manifold $M$ has an almost contact structure if it admits a vector field $\xi$, a tensor field $\varphi$ of type (1,1) and a 1-form $\eta$ satisfying 
\begin{equation}\label{e1}
\eta(\xi)=1,\ \varphi^2=-I+\eta\otimes\xi,\ \varphi(\xi)=0,\ 	\eta\circ\varphi=0.
\end{equation}
\cite{G}A Lorentzian metric $g$ is said to be compatible with the structure $(\varphi, \xi, \eta, g)$ if 
\begin{equation}\label{e01111}
g(\varphi X,\varphi Y)=g(X,Y)+\eta(X)\eta(Y).
\end{equation}
If the manifold $M^{2n+1}$ equipped with an almost-contact structure $(\varphi, \xi, \eta, g)$ and a compatible Lorentzian metric $g$, it is
called a Lorentzian almost-contact metric manifold.
Note that equations (\ref{e1}) and (\ref{e01111}) imply 
\begin{equation}\label{e02}
g(X,\xi) =-\eta(X) \quad \text{and}\quad  g(\xi,\xi) =-1.
\end{equation}
Also, equations (\ref{e01111}) implies 
\begin{equation}\label{e03}
g(X,\varphi Y ) = -g (\varphi X, Y ).
\end{equation}
Let us the four tensors $N^{(1)}$, $N^{(2)}$, $N^{(3)}$ and $N^{(4)}$ in almost contact manifold, which are defined by 
$$N^{(1)}(X,Y)=[\varphi,\varphi](X,Y)+2d\eta(X,Y)\xi,$$
$$N^{(2)}(X,Y)=(\mc{L}_{\varphi X}\eta)Y-(\mc{L}_{\varphi Y}\eta)X,\qquad\ $$
$$N^{(3)}(X)=(\mc{L}_{\xi}\varphi)X,$$
$$N^{(4)}(X)=(\mc{L}_{\xi}\eta)X.$$
An almost contact manifold is normal if and only if $N^{(1)}=N^{(2)}=N^{(3)}=N^{(4)}=0$.\\
In a Lorentzian almost contact metric manifold $(M^{2n+1},\varphi, \xi, \eta, g)$, the fundamental 2-form $\Phi$ is defined as
$$\Phi(X, Y) = g(X,\varphi Y)\quad \text{for all}\quad X, Y \in\mf{X}(M).$$
In 1985, Oubi\~{n}a \cite{OJA} proved that an almost contact metric manifold $(M^{2n+1},\\ \varphi, \xi, \eta, g)$ is trans-Sasakian if and only if it is normal and 
\begin{equation}\label{e04}
d\Phi=2\b\eta\wedge\Phi,\quad d\eta=\a\Phi.
\end{equation}
In a Lorentzian almost trans-Sasakian manifold $(M,\varphi, \xi, \eta, g,\a,\b)$ where $\a$ and $\b$ are smooth function on $M$, we have  
\begin{equation}\label{e2}
	(\nabla_X\varphi)Y=\a[\eta(Y)X+g(X,Y)\xi]-\b[\eta(Y)\varphi X+g(\varphi X,Y)\xi],
	\end{equation}
 \begin{equation}\label{e10}
\nabla_X\xi=\a\varphi X +\b(X-\eta(X)\xi),
\end{equation}
\begin{equation}\label{e11}
(\nabla_X\eta)Y=\a g( X,\varphi Y)-\b g(\varphi X, \varphi Y).
\end{equation}
A Lorentzian trans-Sasakian manifold $M$ of constant $\varphi$-holomorphic sectional curvature is called a Lorentzian trans-Sasakian space form.

 \section{Lorentzian trans-Sasakian space form }
 
In this section, we have first derived the expression of curvature tensor on Lorentzian trans-Sasakian space forms with respect to Levi-Civita connection.
\begin{theorem}
    In a Lorentzian trans-Sasakian Manifold $(M,\varphi, \xi, \eta, g,\a,\b)$ (with $\a,\b$ constants) of curvature tensor $R$, 
    \begin{equation}\label{tcr1}
        R(X,Y) \xi=(\a^2-\b^2)[\eta(Y)X-\eta(X)Y]+2\a\b[\eta(X)\varphi Y-\eta(Y)\varphi X].
    \end{equation}
\end{theorem}
\begin{proof}
    From the definition of curvature tensor, we have
         $$
             R(X,Y) Z=\nabla_X\nabla_Y Z-\nabla_Y\nabla_X Z-\nabla_{[X,Y]} Z.
         $$
         Replacing $Z$ by $\xi$, we get
         $$
             R(X,Y) \xi=\nabla_X\nabla_Y \xi-\nabla_Y\nabla_X \xi-\nabla_{[X,Y]} \xi.
         $$
Using (\ref{e10}), we obtain
\begin{eqnarray*}
    R(X,Y) \xi=\a(\nabla_X\varphi) Y-\a(\nabla_Y\varphi) X +\b[(\nabla_Y\eta)X]\xi\\ -\b[(\nabla_X\eta)Y]\xi+\b\eta(X)\nabla_Y\xi 
-\b\eta(Y)\nabla_X\xi.
\end{eqnarray*}
Using (\ref{e2}), (\ref{e10}) and (\ref{e11}), we get
\begin{eqnarray*}
    R(X,Y) \xi=\a^2\eta(Y)X+\a^2 g(X,Y)\xi-\a\b\eta(Y)\varphi X-\a\b g(\varphi X,Y)\xi\\-\a^2\eta(X)Y-\a^2g(X,Y)\xi+\a\b\eta(X)\varphi Y+\a\b g(\varphi Y,X)\xi\\+\a\b g( Y,\varphi X)\xi-\b^2 g(\varphi X, \varphi Y)\xi  -\a\b g( X,\varphi Y)\xi\\+\b^2 g(\varphi X, \varphi Y)\xi+\a\b\eta(X)\varphi Y +\b^2\eta(X)Y-\b^2\eta(X)\eta(Y)\xi\\-\a\b\eta(Y)\varphi X -\b^2\eta(Y)X+\b^2\eta(Y)\eta(X)\xi.
\end{eqnarray*}
Simplifying,
$$R(X,Y) \xi=(\a^2-\b^2)[\eta(Y)X-\eta(X)Y]+2\a\b[\eta(X)\varphi Y-\eta(Y)\varphi X].$$
\end{proof}

\begin{corollary}
     In a Lorentzian trans-Sasakian manifold $(M,\varphi, \xi, \eta, g,\a,\b)$ (with $\a,\b$ constants) of curvature tensor $R$, 
    \begin{equation}\label{tcr2}
        R(\xi,X)Y=-(\a^2-\b^2)[\eta(Y)X+g(X,Y)\xi]-2\a\b[\eta(Y)\varphi X-g(X,\varphi Y)\xi].
    \end{equation}
\end{corollary}
\begin{proof}
    Taking inner product of (\ref{tcr1}) with $Z$, we obtain
       $$R(\xi,Z)X=-(\a^2-\b^2)[\eta(X)Z+g(X,Z)\xi]-2\a\b[\eta(X)\varphi Z-g(\varphi X,Z)\xi].$$
    Replacing $Z$ by $X$ and $X$ by $Y$, we have
    $$R(\xi,X)Y=-(\a^2-\b^2)[\eta(Y)X+g(X,Y)\xi]-2\a\b[\eta(Y)\varphi X-g(X,\varphi Y)\xi].$$
\end{proof}

\begin{lemma}
	Let $(M,\varphi,\xi,\eta,g,\a,\b)$ be a trans-Sasakian manifold of type\\ $(\a,\b)$ and $R$ its curvature tensor. Then for vector fields $X,Y,Z$ orthogonal to $\xi$, we obtain the following expressions,
    
     1.\begin{equation}\label{t01}
        \begin{split}
            R(X,Y)\varphi Z-\varphi R(X,Y)Z=\a^2[g(Y, Z )\varphi X-g(X, Z )\varphi Y+ g( X,\varphi Z)Y\\- g( Y,\varphi Z)X]
             +2\a\b[ g(Y, Z )X-g(X, Z )Y+g(Y,\varphi Z)\varphi X\\-  g( X,\varphi Z)\varphi Y ]
             +\b^2[g( Y, \varphi Z )X-g( X, \varphi Z )Y+ g( X, Z)\varphi Y\\- g( Y, Z)\varphi X].
        \end{split}
        \end{equation}
     2. \begin{equation}\label{t02}
            \begin{split}
                  R(\varphi X,\varphi Y) Z- R(X,Y)Z=\a^2[g(Y,Z)X-g(X,Z)Y+g( Z,\varphi X)\varphi Y\\- g(Z,\varphi Y)\varphi X]
                  +2\a\b[ g( Y,Z)\varphi X -g(X, Z )\varphi Y+ g(Z,\varphi Y)X\\-  g( Z,\varphi X)Y ]+\b^2[ g( X, Z)Y- g( Z, Y)X+g(Z,\varphi Y)\varphi X\\-g( Z,\varphi X )\varphi Y].
            \end{split}
        \end{equation}
     3.\begin{equation}\label{t03}
            \begin{split}
                R(X,Y,\varphi X,\varphi Y)- R(X,Y,X,Y)=(\a^2-\b^2)[g( X, Y)^2-g(X, X )g( Y, Y)\\+g(X,\varphi Y)^2].
            \end{split}
        \end{equation}
     4.\begin{equation}\label{t04}
            \begin{split}
                R(X,\varphi X,Y,\varphi Y)=R(X,\varphi Y,Y,\varphi X)+R(X, Y,X,Y)+(\a^2-\b^2)[g( X, Y)^2\\-g(X, X )g( Y, Y)+g(X,\varphi Y)^2].
            \end{split}
        \end{equation}
     5.\begin{equation}\label{t05}
            \begin{split}
                R(X,\varphi Y,X,\varphi Y)- R(X,\varphi Y,Y,\varphi X)=-(\a^2-\b^2)[g(X, Y )^2+ g( X,\varphi Y)^2\\-g(X, X) g( Y, Y)].
            \end{split}
        \end{equation}
     6.\begin{equation}\label{t06}
            \begin{split}
                            R(Y,\varphi X,Y,\varphi X)- R(X,\varphi Y,Y,\varphi X)=-(\a^2-\b^2)[g( X, Y)^2+g( X,\varphi Y)^2\\ -g( X, X)g(Y,Y)].
            \end{split} 
        \end{equation}
\end{lemma}

\begin{theorem}
    Let $(M,\varphi,\xi,\eta,g,\a,\b,c)$ be a Lorentzian trans-Sasakian space form, where $\a,\b$ are constants and $c$ is constant which derived form $\varphi$-sectional curvature. Then the curvature tensor $R$ of $M$ is
$$4R(X, Y)Z=[3(\a^2-\b^2)-c][g(X, Z )Y-g( Y,Z)X]+(\a^2-\b^2+c)[\eta(Z)\{\eta(Y)X$$$$\quad\qquad-\eta(X)Y\}+\{\eta(Y)g(X, Z )-\eta(X)g( Y,Z)\}\xi+g( X,\varphi Z)\varphi Y$$$$\quad-g( Y,\varphi Z)\varphi X +2 g( X,\varphi Y)\varphi Z]+8\a\b[\{\eta(X)g(Y,\varphi Z)$$
    \begin{equation}\label{teq13}
-\eta(Y)g(X,\varphi Z)\}\xi+\eta(Z)\{\eta(X)\varphi Y-\eta(Y)\varphi X\}].
\end{equation}

\end{theorem}
\begin{proof}
	For any vector field $X$ and $Y$ orthogonal to $\xi$, we have 
	\begin{equation}\label{a01}
	R(X,\varphi X,X,\varphi X)=-c[g(X,X)]^2.
	\end{equation}
	Replacing $X$ by $X+Y$ in (\ref{a01}), we have
$$2R(Y,\varphi X,X,\varphi X)+2R(X,\varphi X,X,\varphi Y)+2R(Y,\varphi Y,X,\varphi X)+2R(Y,\varphi Y,Y,\varphi X)$$
$$+2R(X,\varphi Y,Y,\varphi X)+2R(X,\varphi Y,Y,\varphi Y)+R(X,\varphi Y,X,\varphi Y)+R(Y,\varphi X,Y,\varphi X)$$
\begin{equation}\label{a02}
    =-c[4g(X,Y)^2+4g(X,X)g(X,Y)+2g(X,X)g(Y,Y)+4g(X,Y)g(Y,Y)].
\end{equation}
Replacing $X$ by $X-Y$ in (\ref{a01}), we get
$$-2R(X,\varphi X,X,\varphi Y)-2R(X,\varphi X,Y,\varphi X)+2R(X,\varphi X,Y,\varphi Y)-2R(Y,\varphi X,Y,\varphi Y)$$
$$+2R(X,\varphi Y,Y,\varphi X)-2R(X,\varphi Y,Y,\varphi Y)+R(Y,\varphi X,Y,\varphi X)+R(X,\varphi Y,X,\varphi Y)$$
\begin{equation}\label{a03}
    =-c[4g(X,Y)^2-4g(X,X)g(X,Y)-4g(X,Y)g(Y,Y)+2g(X,X)g(Y,Y)]
\end{equation}
Adding (\ref{a02}) and (\ref{a03}), we have
$$2R(X,\varphi X,Y,\varphi Y)+2R(X,\varphi Y,Y,\varphi X)+R(X,\varphi Y,X,\varphi Y)+R(Y,\varphi X,Y,\varphi X)$$
 \begin{equation}\label{a04}
     =-2c[2g(X,Y)^2+g(X,X)g(Y,Y)].
 \end{equation}
Using (\ref{t01}), (\ref{t02}), (\ref{t03}) in (\ref{a04}), we get

 \begin{equation}\label{a05}
   3R(X, \varphi Y,Y,\varphi X)+R(X, Y,X,Y)=-c[2g(X,Y)^2+g(X,X)g(Y,Y)].
\end{equation}
Replacing $Y$ by $\varphi Y$ in (\ref{a05}), we get
$$3R(X, Y,\varphi X,\varphi Y)+R(X,\varphi Y,X,\varphi Y) =-c[2g(X,\varphi Y)^2+g(X,X)g( Y, Y)].$$
Using (\ref{t02}) and (\ref{t04}) in preceding equation, we get
\begin{eqnarray}\label{t06}
    3 R(X,Y,X,Y)+R(X,\varphi Y,Y,\varphi X)+2(\a^2-\b^2)[g( X, Y)^2\nonumber\\-g(X, X )g( Y, Y)+g(X,\varphi Y)^2]=-c[2g(X,\varphi Y)^2+g(X,X)g( Y, Y)].
\end{eqnarray}
Multiplying (\ref{t06}) by 3 and then subtracting (\ref{a05}), we obtain
\begin{eqnarray}\label{a07}
    4 R(X,Y,X,Y)=-[3(\a^2-\b^2)-c]g(X, Y )^2-3[(\a^2-\b^2)+c] g( X,\varphi Y)^2\nonumber\\+[3(\a^2-\b^2)-c]g(X, X) g( Y, Y).\quad\quad\quad\quad
\end{eqnarray}
Replacing $X$ by $X+Z$ in (\ref{a07}) and simplifying, we get

\begin{equation}\label{a08}
    4R(X, Y)Y=[3(\a^2-\b^2)-c][g(X, Y )Y-g( Y, Y)X]+3[(\a^2-\b^2)+c]g( X,\varphi Y)\varphi Y.
\end{equation}
Again replacing $Y$ by $Y+Z$, in (\ref{a08}) and after simplifying, we get
$$4R(X, Z)Y+4R(X, Y)Z=[3(\a^2-\b^2)-c][g(X, Z )Y+g(X, Y )Z-2g( Y, Z)X]$$
\begin{equation}\label{a09}
   \qquad\qquad\qquad\qquad\quad +3[(\a^2-\b^2)+c][g( X,\varphi Z)\varphi Y+g( X,\varphi Y)\varphi Z].
\end{equation}
Replacing $X$ by $Y$ and $Y$ by $-X$, we get
$$4R(X, Y)Z+4R(Z, Y)X=[3(\a^2-\b^2)-c][-g(Y, Z )X-g(X, Y )Z+2g( X, Z)Y]$$
\begin{equation}\label{a010}
   \qquad\qquad\qquad\qquad\quad +3[(\a^2-\b^2)+c][-g( Y,\varphi Z)\varphi X-g( Y,\varphi X)\varphi Z].
\end{equation}
Adding (\ref{a09}) and (\ref{a010}), we get
\begin{eqnarray*}
    8R(X, Y)Z+4\{R(X, Z)Y+R(Z, Y)X\}=3[3(\a^2-\b^2)-c][g(X, Z )Y\\-g( Y, Z)X]+3[(\a^2-\b^2)+c][g( X,\varphi Z)\varphi Y\\+2g( X,\varphi Y)\varphi Z-g( Y,\varphi Z)\varphi X].
\end{eqnarray*}
Using Bianchi's identity in above relation, we have
\begin{eqnarray}\label{a011}
    4R(X, Y)Z=[3(\a^2-\b^2)-c][g(X, Z )Y-g( Y, Z)X]\nonumber\\+(\a^2-\b^2+c)[g( X,\varphi Z)\varphi Y+2g( X,\varphi Y)\varphi Z-g( Y,\varphi Z)\varphi X].
\end{eqnarray}
Now, let $X,$ $Y$ and $Z$ be arbitrary
vectors fields. Then we can write
$$X=X'-\eta(X')\xi, Y=Y'-\eta(Y')\xi\ \text{and}\ Z=Z'-\eta(Z')\xi,$$
where $X',$ $Y'$ and $Z'$ are orthogonal to $\xi$. Then from (\ref{a011}), we have
\begin{eqnarray*}
    4R(X', Y')Z'-4\eta(Y')R(X', \xi)Z'-4\eta(X')R(\xi, Y')Z'+4\eta(X')\eta(Y')R(\xi,\xi)Z'\\ -4\eta(Z')R(X',Y')\xi+4\eta(Y')\eta(Z')R(X',\xi)\xi+4\eta(X')\eta(Z')R(\xi,Y')\xi\\ -4\eta(X')\eta(Y')\eta(Z')R(\xi, \xi)\xi=[3(\a^2-\b^2)-c][g(X',Z')Y'-g(Y',Z')X'\\ +\eta(X')\eta(Z')Y'\eta(Y')\eta(Z')X'+\eta(X')g(Y',Z')\xi-\eta(Y')g(X',Z' )\xi]\\+(\a^2-\b^2+c)[g(X',\varphi Z')\varphi Y'-g(Y',\varphi Z')\varphi X'+2g(X',\varphi Y')\varphi Z'].
\end{eqnarray*}
Using (\ref{tcr1}) and (\ref{tcr2}) we get
\begin{eqnarray*}
    4R(X',Y')Z'=[3(\a^2-\b^2)-c][g(X',Z')Y'-g(Y',Z')X']\\+(\a^2-\b^2+c)[\eta(Z')\{\eta(Y')X'-\eta(X')Y'\}+\{\eta(Y')g(X',Z')-\eta(X')g(Y',Z')\}\xi\\+g(X',\varphi Z')\varphi Y'-g(Y',\varphi Z')\varphi X' +2g( X',\varphi Y')\varphi Z']+8\a\b[\{\eta(X')g(Y',\varphi Z')\\-\eta(Y')g(X',\varphi Z')\}\xi+\eta(Z')\{\eta(X')\varphi Y'-\eta(Y')\varphi X'\}].
\end{eqnarray*}
\end{proof}
Observe that after contraction of the above equation, we obtain 
\begin{eqnarray*}\label{teq14}
          S(Y,Z)=\frac{nc-(3n-4)(\a^2-\b^2)}{2}g(Y,Z)+\frac{n(\a^2-\b^2+c)}{2}\eta(Y)\eta(Z)\\+2\a\b g(Y,\varphi Z).\qquad\qquad\qquad\qquad\qquad\qquad\qquad\qquad\qquad
     \end{eqnarray*}
\begin{corollary}\label{cr1}
     Let $(M,\varphi,\xi,\eta,g,\a,\b,c)$ be a Lorentzian trans-Sasakian space form, where $\a,\b$ are constants and $c$ denotes $\varphi$-sectional curvature constant. Then the above expression is the Ricci curvature tensor $S$ of $M$  
     
\end{corollary}

\section{Hyperbolic Ricci Soliton:}
In this section, we derive the expression of the second order Lie derivative and also find the nature of the hyperbolic Ricci soliton on various conditions.   
\begin{theorem}
     If a $(2n+1)$-dimensional Lorentzian trans-Sasakian space form $(M, g,\varphi,\eta,\xi,c,\a,\b)$ with $\b\ne 0$, admits hyperbolic Ricci soliton then $M$ is $\eta-$Einstein and $\lambda =\frac{1}{4\b}\{2(n-1)(\a^2-\b^2)+\mu\}-\b $. Moreover the soliton is expanding, steady and shrinking according as $\mu>2\{(n+1)\b^2-(n-1)\a^2\},$ $\mu=2\{(n+1)\b^2-(n-1)\a^2\}$ and $\mu<2\{(n+1)\b^2-(n-1)\a^2\}$ respectively. 
\end{theorem}
\begin{proof}
    From the well-know Lie derivative property 
$$(\mc{L}_\xi g)(X,Y)=g(\nabla_X\xi,Y)+g(X,\nabla_Y\xi),$$
we get,
\begin{equation}\label{e01}
    (\mc{L}_\xi g)(X,Y)=2\b (g(X,Y)-\eta(X)\eta(Y)).
\end{equation}
Taking Lie derivative along $\xi$ of the above, we acquire 
\begin{equation}\label{e02}
    (\mc{L}_\xi(\mc{L}_\xi g))(X,Y)=\mc{L}_\xi((\mc{L}_\xi g)(X,Y))-(\mc{L}_\xi g)(\mc{L}_\xi X,Y)-(\mc{L}_\xi g)(X,\mc{L}_\xi Y).
\end{equation}
Now
\begin{eqnarray}\label{e03}
    \mc{L}_\xi((\mc{L}_\xi g)(X,Y))=4\b^2(g(X,Y)-\eta(X)\eta(Y))+2\b(g(\mc{L}_\xi X,Y)\nonumber\\+g(X,\mc{L}_\xi Y)-\mc{L}_\xi(\eta(X)\eta(Y))
\end{eqnarray}
and
\begin{equation}\label{e04}
    (\mc{L}_\xi g)(\mc{L}_\xi X,Y)=2\b(g(\mc{L}_\xi X,Y)-\eta(\mc{L}_\xi X)\eta(Y))
\end{equation}
similarly,
\begin{equation}\label{e05}
    (\mc{L}_\xi g)(X,\mc{L}_\xi Y)=2\b(g(X,\mc{L}_\xi Y)-\eta(\mc{L}_\xi Y)\eta(X)).
\end{equation}
Again, we find
     $$\eta(\mc{L}_\xi X)=g(\nabla_\xi X,\xi)-g(\nabla_X\xi,\xi).$$ 
 Thus,
 \begin{equation}\label{e06}
     \eta(\mc{L}_\xi X)\eta(Y)+\eta(\mc{L}_\xi Y)\eta(X)=\mc{L}_\xi(\eta(X)\eta(Y)).
 \end{equation}
Therefore, applying (\ref{e01}), (\ref{e03}), (\ref{e04}), (\ref{e05}) and (\ref{e06}) to (\ref{e02}), we get
\begin{equation}
   \mc{L}_\xi((\mc{L}_\xi)(X,Y))=4\b^2(g(X,Y)-\eta(X)\eta(Y)).
\end{equation}
Now from the equation of hyperbolic Ricci soliton, we get
\begin{equation}\label{e011}
    2\b^2(g(X,Y)-\eta(X)\eta(Y))+2\lambda\b(g(X,Y)-\eta(X)\eta(Y))+S(X,Y)=\mu g(X,Y)
\end{equation}
which, by plugging $Y=\xi$, we get
\begin{equation}\label{e012}
    S(X,\xi)=2(2\b^2+\lambda\b)\eta(X)-\mu\eta(X).
\end{equation}
Using this result in (\ref{e011}), we obtain
$$\lambda=\frac{1}{4\b}\{2(n-1)(\a^2-\b^2)+\mu\}-\b.$$
From (\ref{e011}) we can conclude that Lorentzian trans-Sasakian space form admitting hyperbolic Ricci soliton is $\eta-$Einstein.
The  soliton is expanding, if
\begin{equation}
    \mu>2\{(n+1)\b^2-(n-1)\a^2\},
\end{equation}
 steady, if
\begin{equation}
      \mu=2\{(n+1)\b^2-(n-1)\a^2\},
\end{equation}
and shrinking, if
\begin{equation}
      \mu<2\{(n+1)\b^2-(n-1)\a^2\}.
\end{equation}
\end{proof}

\section{Example:}
Let us define a three-dimensional manifold $M=\{(x,y,z)\in\mathbb{R}^3,y\neq 0\}$, where $(x,y,z)$ are the standard co-ordinate in $\mathbb{R}^3$. The vector field as defined below,

$\quad\quad\quad\quad\quad\quad\quad\quad e_1=e^{z}\frac{\partial}{\partial x} \quad\quad
e_2=e^{z}\frac{\partial}{\partial y} \quad\quad e_3=\frac{\partial}{\partial z}$,\\
are linearly independent at each point of $M$. We define the Riemannian metric $g$ as
\begin{equation}
     g(e_1,e_1)=g(e_2,e_2)=1,\quad g(e_3,e_3)=-1,\quad g(e_1,e_2)=g(e_1,e_3)=g(e_2,e_3)=0.
\end{equation}
Let $\eta$ be the $1-form$ defined by $\eta(X)=-g(X,\xi)$ for any $X\in\mathfrak{X}(M)$ and $\varphi$ be the $(1,1)$ tensor field defined by 
\begin{equation}
    \varphi(e_1)=e_2, \quad \varphi(e_2)=-e_1, \quad \varphi(e_3)=0.
\end{equation}
Then using the linearity of $\varphi$ and $g$, we have
$$ \eta(\xi)=1,\quad  \varphi^2X=-X+\eta(X)\xi,   \quad g(\varphi X,\varphi Y)=g(X,Y)+\eta(X)\eta(Y),$$
for all $ X,Y\in TM.$ Then for $\xi=e_3$, the structure $(\varphi,\xi,\eta,g)$ defines as Lorentzian structure on $M$. Now, from the definition of the Lie bracket, we found the result
$$[e_1,e_2]=0,\quad[e_2,e_3]=-e_2,\quad[e_1,e_3]=-e_1.$$
Let $\nabla$ be the Levi-Civita connetion of the Riemannian metric $g$ defined by the Koszul's formula which is given by
\begin{eqnarray*}
    2g(\nabla_XY,Z)=Xg(Y,Z)+Yg(X,Z)-Zg(Y,X)-g(X,[Y,Z])\\-g(Y,[X,Z])+g(Z,[X,Y])
\end{eqnarray*}
Using the above formula, we can easily calculate the followings: 
$$\nabla_{e_1}e_1=-e_3\quad\quad\quad \nabla_{e_1}e_2=0\quad\quad\quad\nabla_{e_1}e_3=-e_1$$
$$\nabla_{e_2}e_1=0 \quad\quad\nabla_{e_2}e_2=-e_3 \quad\quad\nabla_{e_2}e_3=-e_2$$
$$\nabla_{e_3}e_1=0 \quad\quad\nabla_{e_3}e_2=0 \quad\quad\nabla_{e_3}e_3=0.$$
It verifies all conditions of the Lorentzian trans-Sasakian manifold. Hence $M$ became a trans-Sasakian manifold of type $(0,-1)$. Now taking the above results and using the formula of the Reimannian curvature tensor $R$, that is $R(X,Y)Z=\nabla_X\nabla_YZ-\nabla_Y\nabla_XZ-\nabla_{[X,Y]}Z$, we get the components of Riemanian curvature tensors as given by \\
$$R(e_1,e_2)e_3=0\quad\quad R(e_2,e_3)e_1=0\quad\quad R(e_1,e_3)e_1=-e_3$$
$$R(e_1,e_2)e_2=e_1\quad\quad R(e_2,e_3)e_2=-e_3\quad\quad R(e_1,e_3)e_3=-e_1$$
$$R(e_1,e_2)e_1=-e_2\quad\quad R(e_2,e_3)e_3=-e_2\quad\quad R(e_1,e_3)e_2=0$$
and the components of Ricci tensors are given by\\
$$S(e_1,e_1)=0\quad\quad\quad S(e_2,e_2)=2\quad\quad\quad S(e_3,e_3)=-2.$$
Hence $M$ is a generalized Lorentzian trans-Sasakian space form with $\phi-$sectional curvature $c=1$.
Let us define a vector field $V$ and putting $V=\xi$. Then we get\\
$$(\mc{L}_Vg)(e_1,e_1)=-2\quad\quad (\mc{L}_Vg)(e_2,e_2)=-2\quad\quad (\mc{L}_Vg)(e_3,e_3)=0,$$
and
$$\mc{L}((\mc{L}_Vg))(e_1,e_1)=4\quad\quad \mc{L}((\mc{L}_Vg))(e_2,e_2)=4\quad\quad \mc{L}((\mc{L}_Vg))(e_3,e_3)=0.$$
Using the above results in the hyperbolic Ricci soliton, we find the relation $\lambda=1+\frac{\mu}{4}$, which verifies the theorem.

\section{ Hyperbolic Conformal Ricci Soliton:}
In this section, we shall introduce a new concept hyperbolic conformal Ricci flow and its soliton. After that we shall prove its existence and derive the condition of hyperbolic conformal Ricci soliton, to expanding, srinking, and steady on an $(2n+1)$-dimensional Lorentzian trans-Sasakian space form with respect to the Levi-Civita connection.\\
Now we introduce the notion of the hyperbolic conformal Ricci flow equation as an evolution equation given by:
\begin{equation}
    \frac{\partial^2g}{\partial t^2}=-[2S+(p+\frac{2}{n})g],
\end{equation}
for $g$, a Riemannian metric on an $n-$dimensional smooth manifold $M$.
\begin{definition}\label{de01}
     A Riemannian manifold $(M,g)$ is said to admit hyperbolic conformal Ricci soliton if there exists a vector field $V$ on $M$ and real scalars $\mu$ and $\lambda$ such that
     \begin{equation}\label{eq006}
         \mc{L}_V(\mc{L}_Vg)+2\lambda\mc{L}_Vg+2S=2(\mu-\frac{1}{2}(p+\frac{2}{n})) g.
     \end{equation}
\end{definition}\noindent
The soliton is said to be expanding, steady or shrinking according as $\lambda$ is positive, zero or negative respectively.\\
Now we shall show that such Riemannian manifold exists in the following lemma.
\begin{lemma}
    Hyperbolic conformal Ricci solitons are self similar solutions of the Hyperbolic conformal Ricci flow.
\end{lemma}
\begin{proof}
    We consider a family of Riemannian metrics
    \begin{equation}\label{ab01}
        \Bar{g}(t)=\sigma(t)\psi_t^*(g(t)),
    \end{equation}
    where $\psi_t^*$ is a 1-parameter family of deffeomorphism function from $M$ to $M$.
    Now differentiating both side with respect to $t$, we get
    \begin{eqnarray*}
        \frac{\partial\Bar{g}}{\partial t}=\sigma'(t)\psi_t^*(g(t))+\sigma(t)\psi_t^*(\frac{\partial g}{\partial t})+\sigma(t)\psi_t^*(\mc{L}_Vg),
    \end{eqnarray*}
    where $V$ is a vector field depend on $t$ on $M$.
    Again, differentiating both side, we have
    \begin{align}\label{ab02}
        \frac{\partial^2\Bar{g}}{\partial t^2}=\sigma''(t)\psi_t^*(g(t))+2\sigma'(t)\psi_t^*(\frac{\partial g}{\partial t})+2\sigma'(t)\psi_t^*(\mc{L}_Vg)+\sigma(t)\psi_t^*(\frac{\partial^2 g}{\partial t^2})\nonumber\\+\sigma(t)\psi_t^*(\mc{L}_V\frac{\partial g}{\partial t})+\sigma(t)\psi_t^*(\frac{\partial (\mc{L}_Vg)}{\partial t})+\sigma(t)\psi_t^*(\mc{L}_V(\mc{L}_V)g).
    \end{align}
    Now, given a metric $g_0$, a vector vield $Y$ and real scalar $\mu~and~\lambda$ satisfies
    \begin{align}\label{ab03}
                 \mc{L}_Y(\mc{L}_Yg_0)+2\lambda\mc{L}_Yg_0-2\mu g_0=-[2S+(p+\frac{2}{n}) g_0].
    \end{align}
    Putting $g(t)=g_0,~\sigma(t)=1+\lambda t-\mu t^2$ and integrating t-dependent vector field $V(t)=\frac{1}{\sigma(t)}Y$, to give a family of diffeomorphisms $\psi_t^*~with~\psi_0$. From equation (\ref{ab02}) the metric $\Bar{g}$ defined by $\Bar{g_0}$, then we get
    \begin{align}
         \frac{\partial^2\Bar{g}}{\partial t^2}=\psi_t^*(\sigma''(t)g_0+2\sigma'(t)\frac{1}{\sigma(t)}\mc{L}_Yg_0+\frac{1}{\sigma(t)}\mc{L}_Y(\mc{L}_Yg_0)).
    \end{align}
    Since the geometry of the metric $g(t)$ is the same as that of $g_0$, we will concentrate on $g_0$. Put $\sigma''(0),\sigma'(0)~and~\sigma(0)$, so the above equation becomes,
    \begin{align}
        \frac{\partial^2\Bar{g}}{\partial t^2}=\psi_t^*(-2\mu g_0+2\lambda\mc{L}_Yg_0+\mc{L}_Y(\mc{L}_Yg_0)).
    \end{align}
    Using (\ref{ab03}) in above equation, we obtain
    \begin{align}
        \frac{\partial^2\Bar{g}}{\partial t^2}=\psi_t^*(-[2S+(p+\frac{2}{n})g_0])\nonumber\\
        =-[2S+(p+\frac{2}{n})\Bar{g}].
    \end{align}
    It shows that hyperbolic conformal Ricci soliton satisfies the hyperbolic conformal Ricci flow equation.
\end{proof}

\begin{theorem}
    If a $(2n+1)$-dimensional Lorentzian trans-Sasakian space form $(M, g,\varphi,\eta,\xi,c,\a,\b)$ with $\b\ne 0$ admits hyperbolic conformal $Ricci$ soliton, then $\lambda =\frac{1}{4\b}[2(n-1)(\a^2-\b^2)+(\mu-\frac{1}{2}(p+\frac{2}{2n+1})]-\b$. Furthermore, the soliton is shrinking, steady or expanding according as;
    $ \mu<2\{(n+1)\b^2-(n-1)\a^2\}+\frac{1}{2}(p+\frac{2}{2n+1})$, $\mu=2\{(n+1)\b^2-(n-1)\a^2\}+\frac{1}{2}(p+\frac{2}{2n+1})$ and $\mu>2\{(n+1)\b^2-(n-1)\a^2\}+\frac{1}{2}(p+\frac{2}{2n+1})$ respectively.
\end{theorem}

\begin{proof}
    Let $(M,g)$ be a $(2n+1)$ dimensional manifold, then from the definition of hyperbolic conformal Ricci soliton, we rewrite the equation as
\begin{equation}\label{e001}
    \frac{1}{2}\mc{L}_\xi(\mc{L}_\xi g)(X,Y)+\lambda\mc{L}_\xi g(X,Y)+S(X,Y)=(\mu-\frac{1}{2}(p+\frac{2}{2n+1})]) g(X,Y).
\end{equation}
Inserting (\ref{e01}) and (\ref{e06}) in (\ref{e001}), we have
\begin{equation}\label{e002}
    S(X,Y)=(\mu-\frac{1}{2}(p+\frac{2}{2n+1})-2\b^2-2\lambda\b) g(X,Y)+2\b(\b+\lambda)\eta(X)\eta(Y).
\end{equation}
So, we see that $M$ is an $\eta-$ Einstein manifold.\\
We substituted $Y=\xi$, we get
$$2(n-1)(\a^2-\b^2)\eta(X)=\{4\b(\b+\lambda)-(\mu-\frac{1}{2}(p+\frac{2}{2n+1})\}\eta(X).$$
Hence, we have
$$\lambda=\frac{1}{4\b}[2(n-1)(\a^2-\b^2)+(\mu-\frac{1}{2}(p+\frac{2}{2n+1})]-\b.$$
The  soliton is expanding, if
\begin{equation}
    \mu>2\{(n+1)\b^2-(n-1)\a^2\}+\frac{1}{2}(p+\frac{2}{2n+1}),
\end{equation}
 steady, if
\begin{equation}
      \mu=2\{(n+1)\b^2-(n-1)\a^2\}+\frac{1}{2}(p+\frac{2}{2n+1}),
\end{equation}
and shrinking, if
\begin{equation}
      \mu<2\{(n+1)\b^2-(n-1)\a^2\}+\frac{1}{2}(p+\frac{2}{2n+1}).
\end{equation}

\end{proof}

\section{Acknowledgments} The author Bidhan Mondal thanks the University Grants Commission Junior Research Fellowship (Id No. 231610077944) for their financial assistance.

\section{Data Availability} Not applicable
\section{Declarations}
\textbf{Compliance with Ethical Standards} Not Applicable

% ------------------------------------------------------------------------

\begin{thebibliography}{10}
      \bibitem{ABNBl}Basu, Nirabhra, and Arindam Bhattacharyya. "Conformal Ricci soliton in Kenmotsu manifold." Global Journal of Advanced Research on Classical and Modern Geometries 4.1 (2015): 15-21.
      
      \bibitem{BSM}Bhati, S. M. "On Weakly Ricci $\phi-$symmetric $\delta-$Lorentzian trans-Sasakian manifolds." Bull. of Mathematical Analysis and Applications 5.1 (2013): 36-43.
      
      \bibitem{BAM}Blaga, Adara M. "On trivial gradient hyperbolic Ricci and gradient hyperbolic Yamabe solitons." Journal of Geometry 115.2 (2024): 26.
     
      \bibitem{BO}Blaga, Adara M., and Cihan Özgür. "Some properties of hyperbolic Yamabe solitons." arXiv preprint arXiv:2310.15814 (2023).

      \bibitem{DB}Blair, David E., and D. E. Blair. Riemannian geometry of contact and symplectic manifolds. Vol. 203. Boston: Birkhäuser, 2010.

     

       \bibitem{CBL}Chow, Bennett, Peng Lu, and Lei Ni. Hamilton’s Ricci flow. Vol. 77. American Mathematical Society, Science Press, 2023.

      \bibitem{UKD}De, U. C., and Krishnendu De. "On Lorentzian trans-Sasakian manifolds." Communications Faculty of Sciences University of Ankara Series A1 Mathematics and Statistics 62.2 (2013): 37-51.

      \bibitem{FHS}Faraji, Hamed, Shahroud Azami, and Ghodratallah Fasihi-Ramandi. "Three dimensional homogeneous hyperbolic Ricci solitons." Journal of Nonlinear Mathematical Physics 30.1 (2023): 135-155.

      \bibitem{FA}Fischer, Arthur E. "An introduction to conformal Ricci flow." Classical and Quantum Gravity 21.3 (2004): S171.

      \bibitem{G}Calvaruso, Giovanni. "Contact Lorentzian manifolds." Differential Geometry and its Applications 29 (2011): S41-S51.

      \bibitem{NAB}  Ganguly, Dipen, Nirabhra Basu, and  Arindam Bhattacharyya. "Existence of conformal Ricci soliton and characteristics of almost conformal Ricci solitons on Sasakian manifold." Journal of Mathematical Sciences 280.2 (2024): 146-156.

      \bibitem{ITM} Ikawa, Toshihiko, and Mehmet Erdogan. "Sasakian manifolds with Lorentzian metric." Kyungpook Mathematical Journal 35.3 (1996): 517-517.

      \bibitem{KK}Kenmotsu, Katsuei. "A class of almost contact Riemannian manifolds." Tohoku Mathematical Journal, Second Series 24.1 (1972): 93-103.

      \bibitem{KL}Kong, De-Xing, and Kefeng Liu. "Wave character of metrics and hyperbolic geometric flow." Journal of mathematical physics 48.10 (2007).

      \bibitem{MK} Matsumoto, Koji. "On Lorentzian paracontact manifolds." Bulletin of Yamagata University. Natural science 12.2 (1989): p151-156.

      \bibitem{OJA}Oubiña, José A. "New Class of almost contact metric structure." Publication Math. Debrecen 32 (1985): 187-193.

      \bibitem{SP}Panda S, Halder K, Bhattacharyya A. *-Solitonon on Lorentzian Kenmotsu space form. Global Journal of Advance Research on Classical and Modern Geometries. 2023 Jan 1;12(1).

       \bibitem{PS}Pujar, S. S. "On $\delta$ Lorentzian $\a$ Sasakian manifolds." Antactica J. of Mathematics 8 (2012).
       
      \bibitem{PSS}Pujar, S. S., and V. J. Khairnar. "On Lorentzian trans-Sasakian manifold-I." Int. J. of Ultra Sciences of Physical Sciences 23.1 (2011): 53-66.

      
	
	   \bibitem{MDS}Siddiqi, Mohd Danish. "Generalized Ricci solitons on trans-Sasakian manifolds." Khayyam Journal of Mathematics 4.2 (2018): 178-186.

     

     \bibitem{YAA}Yaliniz, A. Funda, A. H. M. E. T. Yildiz, and MINE TURAN. "On three-dimensional Lorentzian Kenmotsu manifolds." Kuwait J. Sci. Eng 36.2A (2009): 51-62.
      
    \bibitem{YAM}Yildiz, Ahmet, Mine Turan, and Cengizhan Murathan. "A class of Lorentzian $\a$-Sasakian manifolds." Kyungpook Mathematical Journal 49.4 (2009): 789-799.
 
\end{thebibliography}
\end{document}